\documentclass[12pt,twoside,reqno]{amsart}
\usepackage[colorlinks=true,citecolor=blue]{hyperref}
\usepackage{mathptmx, amsmath, amssymb, amsfonts, amsthm, mathptmx, enumerate, color,mathrsfs}
\usepackage{graphicx}

\usepackage{epstopdf}

\newtheorem{theorem}{Theorem}[section]
\newtheorem{lemma}[theorem]{Lemma}

\newtheorem{corollary}[theorem]{Corollary}

\theoremstyle{definition}

\newtheorem{remark}[theorem]{Remark}
\numberwithin{equation}{section}

\begin{document}
\setcounter{page}{1}

\vspace*{2.0cm}
\title[Gradient estimates and Liouville type theorems]
{Cheng-Yau logarithmic gradient estimates for a nonlinear elliptic equation on smooth metric measure spaces}
\author[C. Jin, Y. Wang, F. Zeng]{Cheng Jin$^1$, Youde Wang$^{2}$, Fanqi Zeng$^{1,*}$ }
\maketitle
\vspace*{-0.6cm}

\begin{center}
{\footnotesize

$^1$School of Mathematics and Statistics, Xinyang Normal University, Xinyang, 464000, P.R. China\\
$^2$School of Mathematics and Information Sciences, Guangzhou University, Guangzhou, China;\\
$^2$ Hua Loo-Keng Key Laboratory of Mathematics, Institute of Mathematics, Academy of Mathematics and Systems Science, Chinese Academy of Sciences, Beijing, 100190, China;\\
$^2$ School of Mathematical Sciences, University of Chinese Academy of Sciences, Beijing, 100049, China
}\end{center}

\vskip 4mm {\footnotesize \noindent {\bf Abstract.}
In this paper, we consider the nonlinear elliptic equation
$$\Delta_fv^\tau+\lambda v=0$$
on a complete smooth metric measure space with $m$-Bakry-\'{E}mery Ricci curvature bounded from below, where $\tau>0$ and $\lambda$ are constant.
We obtain some new local gradient estimates for positive solutions to the equation using the Nash-Moser iteration technique. As applications of these estimates, we obtain a Liouville type theorem and a Harnack inequality, and the global gradient estimates for such solutions. Our results generalize and improve the estimates in Wang (J. Differential Equations 260:567-585, 2016) and Zhao (Arch. Math. (Basel) 114:457-469, 2020).

}.


\renewcommand{\thefootnote}{}
\footnotetext{ $^*$Corresponding author.
\par
E-mail addresses: m18237140062@163.com (C. Jin), wyd@math.ac.cn (Y. Wang), zengfq@xynu.edu.cn (F. Zeng).
}

\section{Introduction}

In this paper, we consider Cheng-Yau type gradient estimates for the positive solutions of the following nonlinear elliptic equation
\begin{equation}\label{zjInt1}
\Delta_fv^{\tau}+\lambda v=0
\end{equation}
on a smooth metric measure space $(M^{n},g,e^{-f}d\nu)$ whose $m$-Bakry-\'{E}mery Ricci tensor
\begin{equation}\label{zjInt2}
{\rm Ric}_{f}^{m}={\rm Ric}+\nabla^{2}f-\frac{1}{m-n}{\rm d}f\otimes {\rm d}f,\quad\quad~~~~ m> n
\end{equation}
satisfies ${\rm Ric}_{f}^{m}\geq-(m-1)Kg$ for some constant $K\geq0$, where $\tau>0$ and $\lambda$ are constant, and $\Delta_{f}\cdot=\Delta -\langle \nabla f, \nabla\cdot\rangle$
is the weighted Laplacian for some smooth potential function $f$.
Here $d\nu$, ${\rm Ric}$, $\Delta$, $\nabla^{2}$ and $\nabla$ are the volume form, Ricci tensor, the usual Laplacian, Hessian and gradient operators associated with the metric $g$ respectively.

When $f$ is constant, $\rho=v^{\tau}$ and $p=\frac{1}{\tau}$, equation \eqref{zjInt1} becomes
\begin{equation}\label{zjInt3}
\Delta \rho+\lambda \rho^{p}=0.
\end{equation}
Equation \eqref{zjInt1} appears in differential geometry. Indeed, for any $n$-dimensional $(n \geq3)$ complete manifold
$(M, g)$, consider a pointwise conformal metric $\tilde{g}=\rho^{\frac{4}{n-2}}g$ for some smooth positive function $\rho$.
Then the scalar curvature $\tilde{R}$ of metric $\tilde{g}$ related to the scalar curvature $R$ of metric $g$ is given by
\begin{equation}\label{zjInt4}
\Delta \rho-\frac{n-2}{4(n-1)}Ru+\frac{n-2}{4(n-1)}\tilde{R}\rho^{\frac{n+2}{n-2}}=0.
\end{equation}
When $R=0$ and $\tilde{R}$ is a constant, \eqref{zjInt4} reduces to \eqref{zjInt3} with $\lambda=\frac{n-2}{4(n-1)}\tilde{R}$ and $\tau=\frac{n-2}{n+2}$.
If $M$ is compact and $\tilde{R}$ is constant, the existence of a positive solution $\rho$ is the well-known Yamabe problem. For more details we refer to \cite{Lee1987}.

Another important reason of studying equation\eqref{zjInt1} is that its parabolic counterpart is related to the porous medium
equation and the fast diffusion equation. It is a nonlinear heat equation in the form
\begin{equation}\label{zjInt5}
\partial_{t}v=\Delta v^{\tau}.
\end{equation}
Equation \eqref{zjInt5} with $\tau>1$ is called the porous medium equation, which describes the flow of an isentropic gas through a porous medium \cite{Muskat1937}.
Equation \eqref{zjInt5} with $0<\tau<1$ is called the fast diffusion equation, which appears in various geometric flows such as the Ricci flow on surfaces \cite{Hamilton1988} and the Yamabe flow \cite{Ye1994}. There is a lot of literature on the study of gradient estimates of the porous medium equation and the fast diffusion equation. We only cite \cite{Caozhu2015,Huangli2013,Huangma2017,Huangxu2016,JiangCheng2019,Luni2019,Zhu2013,Zhu2011,wu2023} here and one may find more references therein.

On the other hand, inspired by the work of Yau \cite{Yau1975} and Li and Wang \cite{LiWang2002}, many gradient estimates for equation \eqref{zjInt1} on Riemannian manifolds were obtained.
In \cite{Wang2016}, Wang established gradient estimates for positive solutions to the equation \eqref{zjInt1} on any smooth metric measure space whose $m$-Bakry-\'{E}mery Ricci tensor is bounded from below by $-(m-1)K$ with $K\geq0$. We state his main result as below.

\begin{theorem}\label{add1Int-1}(\cite{Wang2016})
Let $(M^{n}, g, e^{-f} d\nu)$ be a smooth metric measure space with ${\rm Ric}_{f}^{m}\geq-(m-1)Kg$ for some $K\geq0$ and $m >2$. Suppose that $v$ is a smooth
positive solution to the equation \eqref{zjInt1} and $u=\frac{\tau}{\tau-1}v^{\tau-1}$.

(1) If $\tau>1$, then
\begin{equation}\label{zjInt6}
\lambda\leq C_{1}\left(m,\tau,K,\sup_{M}u\right), \quad\quad \frac{|\nabla u|^{2}}{u^{2}}\leq C_{2}\left(m,\tau,\lambda,K,\sup_{M}u\right),
\end{equation}
where $C_{1} > 0$, $C_{2} \geq 0$ are explicit.

(2) If $1-\frac{2}{m}<\tau<1$, then
\begin{equation}\label{zjInt7}
\lambda\leq \overline{C}_{1}\left(m,\tau,K,\sup_{M}u\right), \quad\quad \frac{|\nabla u|^{2}}{u^{2}}\leq -\overline{C}_{2}\left(m,\tau,\lambda,K,\sup_{M}u\right),
\end{equation}
where $\overline{C}_{1} > 0$, $\overline{C}_{2} \geq 0$ are explicit and $u < 0$.
\end{theorem}
As a corollary, the author deduced when $v$ is non-trivial a lower bound for $\sup_{M} v$ (if $\tau>1$), respectively an upper bound for $\inf_{M} v$ (if $1-\frac{2}{m}<\tau<1$), and consequently a Liouville type theorem. A Harnack inequality was also proved in \cite{Wang2016}.

Recently, Zhao \cite{Zhao2020} studied gradient estimates for positive solutions of the nonlinear elliptic equation
\begin{equation}\label{zjInt8}
\Delta_Vv^{\tau}+\lambda v=0
\end{equation}
on a Riemannian manifold $(M, g)$ with the $k$-Bakry-\'{E}mery Ricci curvature bounded
from below, where $\Delta_V$ denotes a weighted Laplacian associated to a certain
vector field $V$. We state his main result as below.

\begin{theorem}\label{add2Int-1}(\cite{Zhao2020})
Let $(M^{n}, g)$ be a complete Riemannian manifold. Suppose ${\rm Ric}_{V}^{k}\geq-Kg$ for some $K\geq0$ and $k >2$.
Let $v$ be a positive solution to the equation \eqref{zjInt8}. For any $\sigma>\frac{1}{(k-1)^{2}}$, if
$1\leq\tau<1+\frac{1}{(1+\sigma)(k-1)}$, then
\begin{equation}\label{zjInt9}
\sup_{B_{R/2}(o)}\frac{|\nabla v|}{v}\leq C_{3}\left(\sigma,k,\tau,K,\lambda,R,\left(\inf_{B_{R/2}(o)} v\right)^{1-\tau}\right).
\end{equation}
\end{theorem}

The method they employ is the maximum principle. The gradient estimates in Theorem \ref{add1Int-1} and Theorem \ref{add2Int-1} depend on the
bound of the solution. Moreover, the gradient estimates are not Cheng-Yau type.

In \cite{ChengYau1975}, Cheng and Yau proved the well-known Cheng-Yau type gradient estimate for positive harmonic functions.
We state their main result as below.
\begin{theorem}\label{add3Int-1}(\cite{ChengYau1975})
Let $(M^{n}, g)$ be a complete Riemannian manifold. Suppose ${\rm Ric}\geq-Kg$ for some $K\geq0$.
Then for any positive harmonic function $v$ in a geodesic ball $B_{R/2}(o)$, there is a constant $C_4(n)$ depending
only on $n$ such that

\begin{equation}\label{zjInt10}
\sup_{B_{R/2}(o)}\frac{|\nabla v|}{v}\leq C_4(n)\frac{1+\sqrt{K}R}{R}.
\end{equation}
\end{theorem}
We note that Cheng-Yau type gradient estimats depend only on $n$, $K$ and $R$, and it does not depend on some other geometric quantities; moreover, Cheng-Yau type gradient estimats are more significant and useful in geometric analysis, since it can derive a strong Liouville theorem and Harnack inequalities. There have been extensive studies on Cheng-Yau type gradient estimats over the recent decades; see for example \cite{WangZhang2011,WangWei2023,HeWangWei2024,huangguo2024,WangWangWei2024,huWanghe2024} and references therein.

Inspired by these works, we employ the Nash-Moser iteration
technique to prove a Cheng-Yau type gradient estimate for positive solutions to equation\eqref{zjInt1} on any smooth metric measure space with $m$-Bakry-\'{E}mery Ricci tensor bounded from below. Our first main result is in the following.

\begin{theorem}\label{Int-1}
Let $(M^{n}, g, e^{-f} d\nu)$ be a complete $n$-dimensional smooth metric measure space with ${\rm Ric}_{f}^{m}\geq-(m-1)Kg$ for some $K\geq0$ and $m>2$. Suppose that $\tau >0$ and $v$ is a smooth
positive solution to the equation \eqref{zjInt1} on a geodesic ball $B_{R}(o)\subset M$. If $m$, $\lambda$ and $\tau$ satisfy
\begin{equation}\label{Int5}
\lambda \bigg(\frac{m+1}{m-1}-\frac1\tau \bigg)\geq 0,
\end{equation}
or
\begin{equation}\label{Int6}
1-\frac{2+\sqrt{2}}{m+1+\sqrt{2}}\leq \tau \leq 1-\frac{2-\sqrt{2}}{m+1-\sqrt{2}},
\end{equation}
then there holds
\begin{equation}\label{Int10}
\sup_{B_{R/2}(o)}\frac{|\nabla v|}{v}\leq c(m,\tau,\lambda)\frac{1+\sqrt{K}R}{R},
\end{equation}
where the constant $c(m,\tau,\lambda)$ depends only on $m$, $\tau$ and $\lambda$.
\end{theorem}

By carefully analyzing the conditions \eqref{Int5} and \eqref{Int6} in Theorem \ref{Int-1}, it is easy to see that
the following results hold.

\begin{corollary}\label{Int-1-3}
Let $(M^{n}, g,e^{-f} d\nu)$ be a complete $n$-dimensional smooth metric measure space with ${\rm Ric}_{f}^{m}\geq-(m-1)Kg$ for some $K\geq0$. Suppose that $\tau >0$ and $v$ is a smooth
positive solution to the equation \eqref{zjInt1} on a geodesic ball $B_{R}(o)\subset M$. If $m$, $\lambda$ and $\tau$ satisfy
\begin{equation}\label{Int11}
\lambda\geq 0,\quad \tau\geq\ 1-\frac{2}{m+1},
\end{equation}
or
\begin{equation}\label{Int12}
\lambda\leq 0,\quad 0<\tau\leq1-\frac{2-\sqrt{2}}{m+1-\sqrt{2}},
\end{equation}
then we have
$$\sup\limits_{B_{R/2}(o)}\frac{|\nabla v|}{v}\leq c(m,\lambda ,\tau)\frac{1+\sqrt{K}R}{R}.$$
\end{corollary}

\begin{remark} Here we want to give several remarks of the above Cheng-Yau type gradient estimates.

(1) Compared with the previous work (cf. Theorem \ref{add1Int-1} and Theorem \ref{add2Int-1}), Theorem \ref{Int-1}
and Corollary \ref{Int-1-3} extend $\tau$ to a larger range.

(2) In Theorem \ref{Int-1} and Corollary \ref{Int-1-3}, the gradient estimates on a positive solution $v$ do
not involve the bound of $v$. Meanwhile, the gradient estimates obtained in Theorem \ref{add1Int-1} and Theorem \ref{add2Int-1} concern the bound of $v$. Hence we improve the previous results in Theorem \ref{add1Int-1} and Theorem \ref{add2Int-1}.

(3) Theorem \ref{Int-1} and Corollary \ref{Int-1-3} provide a united expression of previous gradient estimates in Theorem \ref{add1Int-1} for various ranges of $\tau$.
\end{remark}

As an application of Theorem \ref{Int-1}, we get the following Liouville type theorem.

\begin{corollary}\label{Int-1-4}
Let $(M^{n}, g, e^{-f} d\nu)$ be a complete $n$-dimensional smooth metric measure space with ${\rm Ric}_{f}^{m}\geq 0$. Suppose that $\tau >0$ and $v$ is a smooth
positive solution to the equation \eqref{zjInt1} on a geodesic ball $B_{R}(o)\subset M$. If $m$, $\lambda$ and $\tau$ satisfy
$$\lambda\geq 0,\quad \tau\geq\ 1-\frac{2}{m+1},$$
or
$$\lambda\leq 0,\quad 0<\tau\leq1-\frac{2-\sqrt{2}}{m+1-\sqrt{2}},$$
then $v$ is a positive constant on $M$.
\end{corollary}

\begin{remark}
(1) We extend the range of the corresponding power $\tau$ for the same problem as in \cite{Wang2016}. In addition, comparing
with the Liouville type result derived in \cite{Wang2016}, we don't need the bound of $v$.

(2) The gradient estimate in \cite{Zhao2020} is not Cheng-Yau type, which cannot derive the Liouville
property.
\end{remark}

Another consequence of Theorem \ref{Int-1} is the following Harnack inequality.

\begin{corollary}\label{Int-1-5}
Let $(M^{n}, g, e^{-f} d\nu)$ be a complete $n$-dimensional smooth metric measure space with ${\rm Ric}_{f}^{m}\geq-(m-1)Kg$ for some $K\geq0$. Suppose that $\tau >0$ and $v$ is a smooth
positive solution to the equation \eqref{zjInt1}, defined on a geodesic ball $B_{R}(o)\subset M$, with constants $m$, $\lambda$, $\tau$ satisfying \eqref{Int11} or \eqref{Int12}. Then, for any $x,y\in B_{R/2}(o)$ there holds
$$v(x)\leq e^{c(\tau,m,\lambda)(1+\sqrt{K}R)}v(y).$$
\end{corollary}

Moreover, using these local gradient estimates, we can obtain the global gradient
estimates for positive solutions to the equation \eqref{zjInt1}.

\begin{theorem}\label{adsec5-1}
Let $(M^{n}, g, e^{-f} d\nu)$ be a complete $n$-dimensional smooth metric measure space with ${\rm Ric}_{f}^{m}\geq-(m-1)Kg$ where $K$ is a non-negative constant. Let $v$ be a global positive solution of \eqref{zjInt1} in $M$.

(1) If $m$, $\lambda$ and $\tau$ satisfy \eqref{Int5}, then
$$\frac{|\nabla v|}{v}\leq\frac{(m-1)\sqrt{K}}{\tau}.$$

(2) If $m$, $\lambda$ and $\tau$ satisfy \eqref{Int6}, then
$$\frac{|\nabla v|}{v}\leq\frac{\sqrt{K}}{\tau\sqrt{\frac12\left(\frac{\sqrt{2}-m-1}{m-1}+\frac{1}{\tau}\right)\left(\frac{\sqrt{2}+m
+1}{m-1}-\frac{1}{\tau}\right)}}.$$
\end{theorem}

The structure of the paper is organized as follows. In Section 2, we give a meticulous
estimate of $\Delta_f|\nabla u|^{2\alpha}$ at $u = -\ln v$. We also recall the
Saloff-Coste's Sobolev embedding theorem in section 2. In section 3, we establish a universal
integral estimate on $|\nabla \ln v|^{2\alpha}$ and then use delicately the Nash-Moser iteration to prove main
results of this paper. The proof of Corollary \ref{Int-1-4} and Corollary \ref{Int-1-5} is provided
in Section 4. The last section is devoted to the proof of Theorem \ref{adsec5-1}.

\section{Preliminary}

In this section we need some preliminary results to prove the gradient estimate in \eqref{Int10} of Theorem \ref{Int-1}.

We observe that equation \eqref{zjInt1} is rewritten as follows:
\begin{equation}\label{add1sec2-1}
\Delta_fv+(\tau -1)v^{-1}|\nabla v|^2+\tau^{-1}\lambda v^{2-\tau}=0.
\end{equation}
Now, let $u=-\ln v$. Then with respect to \eqref{add1sec2-1}, $u$ satisfies the following nonlinear elliptic equation:
\begin{equation}\label{sec2-1}
\Delta_fu-\tau |\nabla u|^2-\tau^{-1}\lambda e^{(\tau -1)u}=0.
\end{equation}
Set $h=|\nabla u|^2.$ Then we have
\begin{equation}\label{sec2-2}
\Delta_fu-\tau h-\tau^{-1}\lambda e^{(\tau -1)u}=0.
\end{equation}

Firstly, we need to establish the following lemma.

\begin{lemma}\label{lem2-2}
Let $(M^{n}, g, e^{-f} d\nu)$ be a complete $n$-dimensional smooth metric measure space with ${\rm Ric}_{f}^{m}\geq-(m-1)Kg$ where $K$ is a non-negative constant and $m>2$. Assume that $v$ is a somooth positive solution to the equation \eqref{zjInt1}. If $m$, $\lambda$, $\tau\in W_1\cup W_2$, where
$$\aligned
&W_{1}=\left\{(m,\lambda ,\tau):\lambda \bigg(\frac{m+1}{m-1}-\frac1\tau \bigg)\geq 0\right\};
\\&W_{2}=\left\{(m,\lambda ,\tau):1-\frac{2+\sqrt{2}}{m+1+\sqrt{2}}\leq \tau \leq 1-\frac{2-\sqrt{2}}{m+1-\sqrt{2}}\right\},
\endaligned$$
then there exists constant $\alpha_0>1$ and some $\beta_0>0$ such that the following inequality
\begin{equation}\label{sec2-15}
\Delta_fh^{\alpha_0}\geq 2\alpha_{0}\beta_0h^{\alpha_{0}+1}-2\alpha_{0}(m-1)K h^{\alpha_{0}}-\alpha_{0}a_{1}|\nabla h|h^{\alpha_{0}-\frac12}
\end{equation}
holds point-wisely in $\{x:h(x)>0\}$, where $a_1=(2-\frac{2}{m-1})\tau$ and $\alpha_{0}$, $\beta_{0}$ depend only on $m$, $\lambda$ and $\tau$. The $\alpha_{0}$ and $\beta_{0}$ may be defined differently in different $W_{1}$ and $W_{2}$.
\end{lemma}

\begin{proof}  For any $\alpha>1$, direct computation gives
$$\Delta_fh^{\alpha}=\alpha (\alpha -1)h^{\alpha -2}|\nabla h|^2+\alpha h^{\alpha -1}\Delta_fh.$$
Using the Bochner formula
$$\frac{1}{2}\Delta_fh=|\nabla^2u|^2+\langle\nabla u,\nabla\Delta_fu\rangle+\mathrm{Ric}_f(\nabla u,\nabla u),$$
we obtain
\begin{equation}\label{sec2-3}
\Delta_fh^{\alpha}=\alpha (\alpha -1)h^{\alpha -2}|\nabla h|^2+2\alpha h^{\alpha -1}\left(|\nabla^2u|^2+\langle\nabla u,\nabla\Delta_fu\rangle+\mathrm{Ric}_f(\nabla u,\nabla u)\right).
\end{equation}

Let $\{e_1,e_2,\ldots,e_n\}$ be an orthonormal frame of $TM$ on a domain with $h\neq0$ such that
$e_{1}=\frac{\nabla u}{|\nabla u|}$. If we denote $\nabla u =\Sigma_{i=1}^{n}u_{i}e_i$, it is easy to see
$u_1=|\nabla u|=h^{1/2}$ and $u_{i}=0$ for any $2 \leq i \leq n$.
Then the following identities hold (see (3.1) and (3.3) in \cite{HeWangWei2024}):
\begin{equation}\label{sec2-4}
u_{11}=\frac{1}{2}h^{-1}\langle\nabla u,\nabla h\rangle,
\end{equation}
\begin{equation}\label{sec2-5}
\sum_{i=1}^nu_{1i}^2=\frac{1}{4}h^{-1}|\nabla h|^2,
\end{equation}
\begin{equation}\label{sec2-9}
\sum_{i=2}^{n}u_{ii}=\tau h+\tau^{-1}\lambda e^{(\tau -1)u}-u_{11}+\langle\nabla f,\nabla u\rangle.
\end{equation}

Using the Cauchy inequality, we get
\begin{equation}\label{sec2-7}\aligned
|\nabla^2 u|^2\geq& \sum_{i=1}^nu_{1i}^2+\sum_{i=2}^{n}u_{ii}^2\\
\geq& u_{11}^2+\frac1{n-1}\left(\sum_{i=2}^{n}u_{ii}\right)^2.
\endaligned\end{equation}
Then using the inequality (see (5) in \cite{BarrosA2018})
$$(a+b)^2\geq \frac{a^2}{1+\delta}-\frac{b^2}{\delta},\quad \delta=\frac{m-n}{n-1},$$
we can derive from \eqref{sec2-9} that
\begin{equation}\label{sec2-10}\aligned
\frac{1}{n-1}\left(\sum_{i=2}^{n}u_{ii}\right)^2=&\frac{1}{n-1}\left(\tau h+\tau^{-1}\lambda e^{(\tau -1)u}-u_{11}+\langle\nabla f,\nabla u\rangle\right)^2\\
\geq&\frac{1}{n-1}\bigg(\frac{n-1}{m-1}\left(\tau h+\tau^{-1}\lambda e^{(\tau -1)u}-u_{11}\right)^2-\frac{n-1}{m-n}\langle\nabla f,\nabla u\rangle^2\bigg)\\
=&\frac{1}{m-1}(\tau h+\tau^{-1}\lambda e^{(\tau-1)u}-u_{11})^2-\frac{1}{m-n}\langle\nabla f,\nabla u\rangle^2\\
=&\frac{1}{m-1}\big(\tau^2 h^2+\tau^{-2}\lambda^2e^{2(\tau-1)u}+u_{11}^2+2\lambda he^{(\tau-1)u}\\
&-2\tau hu_{11}-2\tau^{-1}\lambda e^{(\tau-1)u}u_{11}\big)-\frac{1}{m-n}\langle\nabla f,\nabla u\rangle^2.
\endaligned\end{equation}
Then we obtain
\begin{equation}\label{add1sec2-7}\aligned
|\nabla^2 u|^2\geq& u_{11}^2+\frac{1}{m-1}\big(\tau^2 h^2+\tau^{-2}\lambda^2e^{2(\tau-1)u}+u_{11}^2+2\lambda he^{(\tau-1)u}\\
&-2\tau hu_{11}-2\tau^{-1}\lambda e^{(\tau-1)u}u_{11}\big)-\frac{1}{m-n}\langle\nabla f,\nabla u\rangle^2.
\endaligned\end{equation}
We can derive from \eqref{sec2-2} that
\begin{equation}\label{sec2-6}\aligned
\langle\nabla u,\nabla\Delta_fu\rangle&=\langle\nabla u,\nabla(\tau h+\tau^{-1}\lambda e^{(\tau-1)u})\rangle\\
&=\tau\langle\nabla u,\nabla h\rangle+\tau^{-1}(\tau-1)\lambda e^{(\tau-1)u}h.
\endaligned
\end{equation}
Substituting \eqref{sec2-4}, \eqref{sec2-5}, \eqref{sec2-6} and \eqref{add1sec2-7} into \eqref{sec2-3}, we obtain
\begin{equation}\label{sec2-11}\aligned
\frac{h^{1-\alpha}}{2\alpha}\Delta_fh^{\alpha}\geq&\left(2\alpha-1+\frac{1}{m-1}\right)u_{11}^2+\left(2-\frac{2}{m-1}\right)\tau hu_{11}\\
&+\mathrm{Ric}_f^m(\nabla u,\nabla u)+\frac{1}{m-1}\tau^2h^2+\frac{1}{m-1}\tau^{-2}\lambda^2e^{2(\tau-1)u}\\
&+\lambda \left(\frac{m+1}{m-1}-\frac{1}{\tau}\right)he^{(\tau-1)u}-\frac{2}{m-1}\tau^{-1}\lambda e^{(\tau-1)u}u_{11}.
\endaligned\end{equation}
By applying the inequality $a^2+2ab\geq -b^2$, we have
\begin{equation}\label{sec2-12}\aligned
\left(2\alpha-1+\frac{1}{m-1}\right)u_{11}^2-\frac{2}{m-1}\tau^{-1}\lambda e^{(\tau-1)u}u_{11}\\
\geq -\frac{1}{\big((2\alpha-1)(m-1)+1\big)(m-1)}\tau^{-2}\lambda^2 e^{2(\tau-1)u}.
\endaligned\end{equation}
By \eqref{sec2-12} and the assumption that ${\rm Ric}_{f}^{m}\geq-(m-1)Kg$, we obtain
\begin{equation}\label{sec2-13}\aligned
\frac{h^{1-\alpha}}{2\alpha}\Delta_fh^{\alpha}\geq&\left(2-\frac{2}{m-1}\right)\tau hu_{11}-(m-1)Kh+\frac{1}{m-1}\tau^2h^2\\
&+\frac{1}{m-1}\tau^{-2}\lambda^2e^{2(\tau-1)u}+\lambda\left(\frac{m+1}{m-1}-\frac{1}{\tau}\right)he^{(\tau-1)u}\\
&-\frac{1}{\big((2\alpha-1)(m-1)+1\big)(m-1)}\tau^{-2}\lambda^2 e^{2(\tau-1)u}.
\endaligned\end{equation}

To estimate $\Delta_fh^{\alpha}$, we divide the arguments into two cases:

(1) If ($m$, $\lambda$, $\tau$) $\in$  $W_1$, we have
$$\lambda \left(\frac{m+1}{m-1}-\frac{1}{\tau}\right)he^{(\tau-1)u}\geq 0.$$
Furthermore, we observe that
$$\bigg(\frac{1}{m-1}-\frac{1}{\big((2\alpha-1)(m-1)+1\big)(m-1)}\bigg)\tau^{-2}\lambda^2 e^{2(\tau-1)u}\geq 0.$$
Thus, we can always find an $\alpha=\alpha (m,\tau ,\lambda)$ large enough such that
$$\frac{h^{1-\alpha}}{2\alpha}\Delta_fh^{\alpha}\geq\frac{1}{m-1}\tau^2h^2+\left(2-\frac{2}{m-1}\right)\tau hu_{11}-(m-1)Kh.$$

(2) If ($m$, $\lambda$, $\tau$) $\in$  $W_2$, we have
\begin{equation}\label{sec2-14}
\frac{1}{m-1}-\frac{m-1}{2}\left(\frac{m+1}{m-1}-\frac{1}{\tau}\right)^2\geq 0.
\end{equation}
\eqref{sec2-13} can be rewritten as follows:
\begin{equation}\label{add1sec2-13}\aligned
\frac{h^{1-\alpha}}{2\alpha}\Delta_fh^{\alpha}\geq&\left(2-\frac{2}{m-1}\right)\tau hu_{11}-(m-1)Kh+\frac{1}{m-1}\tau^2h^2\\
&+\left(\frac{1}{2(m-1)}\tau^{-2}\lambda^2e^{2(\tau-1)u}+\lambda\left(\frac{m+1}{m-1}-\frac{1}{\tau}\right)he^{(\tau-1)u}\right)\\
&+\bigg(\frac{1}{2(m-1)}-\frac{1}{\big((2\alpha-1)(m-1)+1\big)(m-1)}\bigg)\tau^{-2}\lambda^2 e^{2(\tau-1)u}.
\endaligned\end{equation}
According to $a^2+2ab\geq -b^2$, we get
\begin{equation}\label{add2sec2-13}\aligned
&\frac{1}{2(m-1)}\tau^{-2}\lambda^2e^{2(\tau-1)u}+\lambda\left(\frac{m+1}{m-1}-\frac{1}{\tau}\right)he^{(\tau-1)u}\\
\geq& -\frac{m-1}{2}\left(\frac{m+1}{m-1}-\frac{1}{\tau}\right)^2\tau^2h^2.
\endaligned\end{equation}
Meanwhile, we note that
\begin{equation}\label{add3sec2-13}
\bigg(\frac{1}{2(m-1)}-\frac{1}{((2\alpha-1)(m-1)+1)(m-1)}\bigg)\tau^{-2}\lambda^2 e^{2(\tau-1)u}\geq 0.
\end{equation}
Let
$$B_{m,\tau}=\left(\frac{1}{m-1}-\frac{m-1}{2}\left(\frac{m+1}{m-1}-\frac{1}{\tau}\right)^2\right)\tau^2.$$
Using \eqref{sec2-14}, \eqref{add1sec2-13}, \eqref{add2sec2-13} and \eqref{add3sec2-13}, we can always find an $\alpha=\alpha (m,\tau)$ large enough such that
$$\frac{h^{1-\alpha}}{2\alpha}\Delta_fh^{\alpha}\geq B_{m,\tau}h^2+\left(2-\frac{2}{m-1}\right)\tau hu_{11}-(m-1)Kh.$$

We denote
\begin{equation}\label{add4sec2-40}
\frac{h^{1-\alpha}}{2\alpha}\Delta_fh^{\alpha}\geq B_{m,\tau}h^2+\left(2-\frac{2}{m-1}\right)\tau hu_{11}-(m-1)Kh.
\end{equation}
and
$$\beta_0:=\beta_0(m,\tau)=\begin{cases}\frac{\tau^2}{m-1},&(m,\lambda ,\tau)\text{ lies in }W_1;\\B_{m,\tau},&(m,\lambda ,\tau)\text{ lies in }W_2\setminus W_1.\end{cases}$$
Thus we finish the proof.
\end{proof}

To prove Theorem \ref{Int-1}, we also need the
Saloff-Coste's Sobolev embedding theorem.

\begin{lemma}\label{lem2-0} (\cite{Saloff1992, Dung2016, Zhao18})
Let $(M^{n}, g, e^{-f} d\nu)$ be a complete $n$-dimensional smooth metric measure space with ${\rm Ric}_{f}^{m}\geq-(m-1)Kg$ where $K$ is a non-negative constant. For $m>2$, there exists some positive constant $c(m)$ depending only on $m$, such that for all $B\subset M$ of radius $R$, we have for $\phi\in C_0^\infty (B)$
$$\left(\int_B|\phi|^{\frac{2m}{m-2}}e^{-f}d\nu\right)^{\frac{m-2}{m}}\leq e^{c(m)(1+\sqrt{K}R)}V_f^{-\frac{2}{m}}R^2\left(\int_B|\nabla \phi|^2+R^{-2}\phi^2\right)e^{-f}d\nu,$$
where $V_f=\int_Be^{-f}d\nu$.
\end{lemma}

\section{Proof of Theorem \ref{Int-1}}
In this section, we finish the proof of Theorem \ref{Int-1}.
We proceed with following integral inequality on the solutions to equation \eqref{zjInt1}:

\begin{lemma}\label{lem3-1}
Let $(M^{n}, g, e^{-f} d\nu)$ be a complete $n$-dimensional smooth metric measure space with ${\rm Ric}_{f}^{m}\geq-(m-1)Kg$ where $K$ is a non-negative constant and $\Omega=B_R(o)\subset M$ be a geodesic ball. Suppose that $v$ is a positive solution to the equation \eqref{zjInt1} with the constants $m>2$, $\lambda$ and $\tau$ which satisfy the condition in Theorem \ref{Int-1} (or Corollary \ref{Int-1-3}), $u= -\ln v$ and $h=|\nabla u|^2$. Then there exist constants $a_2$, $a_3$ and $a_4$ depending only on $m$, $\lambda$ and $\tau$ such that for any $t\geq t_0,$ where $t_0$ is defined in \eqref{sec3-12}, there holds
\begin{equation}\label{sec3-14}\aligned
&\beta_{0}\int_{\Omega}h^{\alpha_{0}+t+1}\eta^{2}e^{-f}d\nu+\frac{a_{2}}{t}e^{-t_{0}}V_f^{\frac{2}{m}}R^{-2}
\left(\int_{\Omega}|h^{\frac{\alpha_{0}+t}{2}}\eta|^{\frac{2m}{m-2}}e^{-f}d\nu\right)^{\frac{m-2}{m}}\\
\leq&\:a_{4}t_{0}^{2}R^{-2}\int_{\Omega}h^{\alpha_{0}+t}\eta^{2}e^{-f}d\nu+\frac{a_{3}}{t}\int_{\Omega}h^{\alpha_{0}+t}|\nabla\eta|^{2}e^{-f}d\nu,
\endaligned\end{equation}
where $\eta \in C_0^\infty (\Omega ,\mathbb{R}).$
\end{lemma}

\begin{proof} We set $U=\{x\in\Omega|h(x)=0\}$.
We note that $v$ is smooth on $\Omega\setminus U$. Then for a nonnegative function $\psi$ with compact support in $\Omega\setminus U$, we have
\begin{equation}\label{add53-14}\aligned
&-\int_{\Omega}\alpha_{0}h^{\alpha_{0}-1}\langle\nabla h,\nabla\psi\rangle e^{-f}d\nu \\
\geq &2\beta_{0}\alpha_{0}\int_{\Omega}h^{\alpha_{0}+1}\psi e^{-f}d\nu -2(m-1)\alpha_{0}K\int_{\Omega}h^{\alpha_{0}}\psi e^{-f}d\nu \\
&-a_{1}\alpha_{0}\int_{\Omega}h^{\alpha_{0}-\frac{1}{2}}|\nabla h|\psi e^{-f}d\nu.
\endaligned\end{equation}
For constants $\epsilon>0$ and $t>1$, we choose
$$\psi=h_{\epsilon}^{t}\eta^2,$$
where $h_{\epsilon}=(h-\epsilon)^{+}$, $\eta \in C_0^\infty (\Omega ,\mathbb{R})$ is nonnegative and less than or equal to 1, and $t$ is to be determined later. Then a direct calculation shows that
$$\nabla\psi=th_{\epsilon}^{t-1}\eta^{2}\nabla h+2h_{\epsilon}^{t}\eta\nabla\eta.$$
Inserting this identity into \eqref{add53-14}, we obtain
\begin{equation}\label{sec3-1}\aligned
&2\beta_{0}\alpha_{0}\int_{\Omega}h^{\alpha_{0}+1}h_{\epsilon}^{t}\eta^{2}e^{-f}d\nu+
\int_{\Omega}\alpha_{0}th^{\alpha_{0}-1}h_{\epsilon}^{t-1}|\nabla h|^{2}\eta^{2}e^{-f}d\nu \\
\leq&2(m-1)K\alpha_{0}\int_{\Omega}h^{\alpha_{0}}h_{\epsilon}^{t}\eta^{2}e^{-f}d\nu-\int_{\Omega}2\alpha_{0}h^{\alpha_0-1}h_{\epsilon}^{t}\langle\nabla h,\nabla\eta\rangle\eta e^{-f}d\nu\\
&+a_{1}\alpha_{0}\int_{\Omega}h^{\alpha_{0}-\frac{1}{2}}h_{\epsilon}^{t}|\nabla h|\eta^{2}e^{-f}d\nu.
\endaligned
\end{equation}
Since
\begin{equation}\label{sec3-2}
h_\epsilon^t\langle\nabla h,\nabla\eta\rangle \geq -h_\epsilon^t|\nabla h||\nabla\eta|,
\end{equation}
inserting \eqref{sec3-2} into \eqref{sec3-1} and letting $\epsilon\to0$, we have
\begin{equation}\label{sec3-3}
\begin{aligned}&2\beta_{0}\int_{\Omega}h^{\alpha_{0}+t+1}\eta^{2}e^{-f}d\nu+\int_{\Omega}th^{\alpha_{0}+t-2}|\nabla h|^{2}\eta^{2}e^{-f}d\nu\\
\leq &2(m-1)K\int_{\Omega}h^{\alpha_{0}+t}\eta^{2}e^{-f}d\nu+a_{1}\int_{\Omega}h^{\alpha_{0}+t-\frac{1}{2}}|\nabla h|\eta^{2}e^{-f}d\nu\\
&+2\int_{\Omega}h^{\alpha_{0}+t-1}|\nabla h||\nabla\eta|\eta e^{-f}d\nu.\end{aligned}
\end{equation}
Now applying Young's inequality, we get
\begin{equation}\label{sec3-4}
\begin{aligned}a_{1}h^{\alpha_{0}+t-\frac{1}{2}}|\nabla h|\eta^{2}&\leq\frac{t}{4}h^{\alpha_{0}+t-2}|\nabla h|^{2}\eta^{2}+\frac{a_{1}^{2}}{t}h^{\alpha_{0}+t+1}\eta^{2},\\2h^{\alpha_{0}+t-1}|\nabla h||\nabla\eta|\eta&\leq\frac{t}{4}h^{\alpha_{0}+t-2}|\nabla h|^{2}\eta^{2}+\frac{4}{t}h^{\alpha_{0}+t}|\nabla\eta|^{2}.\end{aligned}
\end{equation}
Now we choose $t$ large enough such that
\begin{equation}\label{sec3-5}
\frac{a_1^2}{t}\leq\beta_0.
\end{equation}
It can be concluded from \eqref{sec3-3}, \eqref{sec3-4} and \eqref{sec3-5} that
\begin{equation}\label{sec3-6}\aligned
&\beta_0\int_\Omega h^{\alpha_0+t+1}\eta^2e^{-f}d\nu+\frac{t}{2}\int_\Omega h^{\alpha_0+t-2}|\nabla h|^2\eta^2 e^{-f}d\nu\\
\leq&2(m-1)K\int_{\Omega}h^{\alpha_{0}+t}\eta^{2}e^{-f}d\nu+\frac{4}{t}\int_{\Omega}h^{\alpha_{0}+t}|\nabla\eta|^{2}e^{-f}d\nu.
\endaligned
\end{equation}
Using
\begin{equation}\label{sec3-7}\aligned
\left|\nabla\left(h^{\frac{\alpha_{0}+t}{2}}\eta\right)\right|^{2}&\leq2\left|\nabla h^{\frac{\alpha_{0}+t}{2}}\right|^{2}\eta^{2}+2h^{\alpha_{0}+t}|\nabla\eta|^{2} \\
&=\frac{(\alpha_0+t)^2}2h^{\alpha_0+t-2}|\nabla h|^2\eta^2+2h^{\alpha_0+t}|\nabla\eta|^2,\endaligned
\end{equation}
we have by \eqref{sec3-6}
\begin{equation}\label{sec3-8}\aligned
&\beta_{0}\int_{\Omega}h^{\alpha_{0}+t+1}\eta^{2}e^{-f}d\nu+\frac{4t}{(2\alpha_{0}+2t)^{2}}\int_{\Omega}\left|\nabla\left(h^{\frac{\alpha_{0}+t}{2}}\eta\right)\right|^{2}e^{-f}d\nu \\
\leq &2(m-1)K\int_{\Omega}h^{\alpha_{0}+t}\eta^{2}e^{-f}d\nu+\frac{4}{t}\int_{\Omega}h^{\alpha_{0}+t}|\nabla\eta|^{2}e^{-f}d\nu \\
&+\frac{8t}{(2\alpha_{0}+2t)^{2}}\int_{\Omega}h^{\alpha_{0}+t}|\nabla\eta|^{2}e^{-f}d\nu.
\endaligned
\end{equation}

In what follows, $a_i (i=1, 2, 3,\cdots)$ stand for positive constants depending on $m$, $\tau$ and $\lambda$,
and $t$ is chosen in such a way that
\begin{equation}\label{sec3-9}
\frac{a_2}t\leq\frac{4t}{(2\alpha_0+2t)^2}\quad\mathrm{and}\quad\frac{8t}{(2\alpha_0+2t)^2}+\frac{4}{t}\leq\frac{a_3}t.
\end{equation}
It can be concluded from \eqref{sec3-8} and \eqref{sec3-9} that
\begin{equation}\label{sec3-10}\aligned
&\beta_0\int_\Omega h^{\alpha_0+t+1}\eta^2e^{-f}d\nu+\frac{a_2}{t}\int_\Omega\left|\nabla\left(h^{\frac{\alpha_0+t}{2}}\eta\right)\right|^2e^{-f}d\nu \\
\leq&2(m-1)K\int_{\Omega}h^{\alpha_{0}+t}\eta^{2}e^{-f}d\nu+\frac{a_{3}}{t}\int_{\Omega}h^{\alpha_{0}+t}\left|\nabla\eta\right|^{2}e^{-f}d\nu.
\endaligned
\end{equation}
Now letting $\phi=h^{\frac{\alpha_{0}+t}{2}}\eta$ in Lemma \ref{lem2-0}, we have
$$\begin{aligned}&e^{-c(m)(1+\sqrt{K}R)}V_f^{\frac{2}{m}}R^{-2}\left(\int_{\Omega}|h^{\frac{\alpha_{0}+t}{2}}\eta|^{\frac{2m}{m-2}}e^{-f}d\nu\right)^{\frac{m-2}{m}} \\
\leq &\int_\Omega\left|\nabla\left(h^{\frac{\alpha_0+t}{2}}\eta\right)\right|^2e^{-f}d\nu+R^{-2}\int_\Omega h^{\alpha_0+t}\eta^2e^{-f}d\nu.\end{aligned}$$
Then combining the above inequality and \eqref{sec3-10}, we can deduce that
\begin{equation}\label{sec3-11}\aligned
&\beta_{0}\int_{\Omega}h^{\alpha_{0}+t+1}\eta^{2}e^{-f}d\nu
+\frac{a_{2}}{t}e^{-c(m)(1+\sqrt{K}R)}V_f^{\frac{2}{m}}R^{-2}\left(\int_{\Omega}|h^{\frac{\alpha_{0}+t}{2}}\eta|^{\frac{2m}{m-2}}e^{-f}d\nu\right)^{\frac{m-2}{m}} \\
\leq &2(m-1)K\int_\Omega h^{\alpha_0+t}\eta^2e^{-f}d\nu+\frac{a_3}{t}\int_\Omega h^{\alpha_0+t}|\nabla\eta|^2e^{-f}d\nu \\
&+\frac{a_2}{t}\int_\Omega R^{-2}h^{\alpha_0+t}\eta^2e^{-f}d\nu.
\endaligned
\end{equation}
Now we let
\begin{equation}\label{sec3-12}
t_{0}=c_{m,\lambda ,\tau}\left(1+\sqrt{K}R\right)\quad\mathrm{and}\quad c_{m,\lambda ,\tau}=\operatorname*{max}\left\{c(m)+1,\frac{a_{1}^{2}}{\beta_0}\right\}.
\end{equation}
Thus, we can choose $t$ large enough such that, for any $t\geq t_{0}$, there holds true
$$2(m-1)K R^2\leq\frac{2(m-1)}{c_{m,\lambda ,\tau}^2}t_0^2\quad\mathrm{and}\quad\frac{a_2}{t}\leq\frac{a_2}{c_{m,\lambda ,\tau}}. $$
Furthermore, there exists $a_4=a_4(m,\lambda ,\tau)>0$ such that
\begin{equation}\label{sec3-13}
2(m-1)K R^{2}+\frac{a_{2}}{t}\leq a_{4}t_{0}^{2}
=a_{4}c_{m,\lambda ,\tau}^{2}\left(1+\sqrt{K}R\right)^{2}.
\end{equation}
Hence, it follows from \eqref{sec3-11} and \eqref{sec3-13} that
$$\aligned
&\beta_0\int_\Omega h^{\alpha_0+t+ 1}\eta^2e^{-f}d\nu+\frac{a_2}{t}e^{-t_0}V_f^{\frac{2}{m}}R^{-2}\left(\int_{\Omega}|h^{\frac{\alpha_{0}+t}{2}}\eta|^{\frac{2m}{m-2}}e^{-f}d\nu\right)^{\frac{m-2}{m}} \\
\leq &a_{4}t_{0}^{2}R^{-2}\int_{\Omega}h^{\alpha_{0}+t}\eta^{2}e^{-f}d\nu+\frac{a_{3}}{t}\int_{\Omega}h^{\alpha_{0}+t}|\nabla\eta|^{2}e^{-f}d\nu.
\endaligned$$
This is the required inequality and we finish the proof of this lemma.
\end{proof}

Using inequality \eqref{sec3-14} in Lemma \ref{lem3-1}, we will get a local estimate of $h$ stated in the following lemma.

\begin{lemma}\label{lem3-2}
Let $(M^{n}, g, e^{-f} d\nu)$ be a complete $n$-dimensional smooth metric measure space with ${\rm Ric}_{f}^{m}\geq-(m-1)Kg$ where $K$ is a non-negative constant and $\Omega=B_R(o)\subset M$ be a geodesic ball. Suppose that $h$ is the same as in the Lemma \ref{lem3-1}. If the constants $m>2$, $\lambda$ and $\tau$ satisfy the condition in Theorem \ref{Int-1} (or Corollary \ref{Int-1-3}), then, for $\beta = \left ( \alpha _0+ t_0\right ) \frac m{m-2}$, there exists a nonnegative constant $a_7= a_7(m,\lambda ,\tau ) > 0$ such that
\begin{equation}\label{sec3-15}
\left(\int_{B_{3R/4}(o)}h^{\beta}e^{-f}d\nu\right)^{\frac{1}{\beta}}\leq a_7V_f^{\frac{1}{\beta}}\frac{t_0^2}{R^2}.
\end{equation}
\end{lemma}

\begin{proof}  Now letting $t=t_0$ in \eqref{sec3-14}, we then have
\begin{equation}\label{sec3-16}\aligned
&\beta_0\int_\Omega h^{\alpha_0+t_0+1}\eta^2e^{-f}d\nu +\frac{a_2}{t_0}e^{-t_0}V_f^{\frac{2}{m}}R^{-2}\left(\int_{\Omega}|h^{\frac{\alpha_{0}+t_{0}}{2}}\eta|^{\frac{2m}{m-2}}e^{-f}d\nu\right)^{\frac{m-2}{m}}\\
\leq&a_{4}t_{0}^{2}R^{-2}\int_{\Omega}h^{\alpha_{0}+t_{0}}\eta^{2}e^{-f}d\nu+\frac{a_{3}}{t_{0}}\int_{\Omega}h^{\alpha_{0}+t_{0}}|\nabla\eta|^{2}e^{-f}d\nu.
\endaligned\end{equation}
Now we let $D=\left\{x\in\Omega|h(x)\geq\frac{2a_{4}t_{0}^{2}}{\beta_{0}R^{2}}\right\}$. Hence, we have
\begin{equation}\label{sec3-17}\aligned
&a_{4}t_{0}^{2}R^{-2}\int_{\Omega}h^{\alpha_{0}+t_{0}}\eta^{2}e^{-f}d\nu\\
=&a_{4}t_{0}^{2}R^{-2}\int_{D}h^{\alpha_{0}+t_{0}}\eta^{2}e^{-f}d\nu
+a_{4}t_{0}^{2}R^{-2}\int_{\Omega\backslash D}h^{\alpha_{0}+t_{0}}\eta^{2}e^{-f}d\nu\\
\leq&\frac{\beta_{0}}{2}\int_{\Omega}h^{\alpha_{0}+t_{0}+1}\eta^{2}e^{-f}d\nu+\frac{a_{4}t_{0}^{2}}{R^{2}}\left(\frac{2a_{4}t_{0}^{2}}
{\beta_{0}R^{2}}\right)^{\alpha_{0}+t_{0}}V_{f}.
\endaligned\end{equation}
Applying \eqref{sec3-17} to \eqref{sec3-16}, we obtain
\begin{equation}\label{sec3-18}\aligned
&\frac{\beta_{0}}{2}\int_{\Omega}h^{\alpha_{0}+t_{0}+1}\eta^{2}e^{-f}d\nu+\frac{a_{2}}{t_{0}}e^{-t_{0}}V_f^{\frac{2}{m}}R^{-2}
\left(\int_{\Omega}|h^{\frac{\alpha_{0}+t_{0}}{2}}\eta|^{\frac{2m}{m-2}}e^{-f}d\nu\right)^{\frac{m-2}{m}}\\
\leq&\frac{a_{4}t_{0}^{2}}{R^{2}}\left(\frac{2a_{4}t_{0}^{2}}{\beta_{0}R^{2}}\right)^{\alpha_{0}+t_{0}}V_{f}+\frac{a_{3}}{t_{0}}\int_{\Omega}h^{\alpha_{0}
+t_{0}}|\nabla\eta|^{2}e^{-f}d\nu.\endaligned
\end{equation}
We choose $\zeta\in C_{0}^{\infty}(B_{R}(o))$ satisfying
$$\begin{cases}0\leq\zeta(x)\leq1,\quad |\nabla\zeta(x)|\leq\frac{C}{R},\quad \forall x\in B_R(o);\\
\zeta(x)\equiv1,\hspace{107pt} \forall x\in B_{3R/4}(o),\end{cases}$$
and choose $\eta=\zeta^{\alpha_{0}+t_{0}+1}.$ Then, we get
\begin{equation}\label{sec3-19}
\aligned
a_3R^2|\nabla\eta|^2\leq a_3C^2\left(\alpha_0+t_0+1\right)^2\eta^{\frac{2\alpha_0+2t_0}{\alpha_0+t_0+1}}\leq a_5t_0^2\eta^{\frac{2\alpha_0+2t_0}{\alpha_0+t_0+1}}.\endaligned
\end{equation}
According to \eqref{sec3-19} and the H\"{o}lder inequality, we obtain
\begin{equation}\label{sec3-20}
\aligned
&\frac{a_{3}}{t_{0}}\int_{\Omega}h^{\alpha_{0}+t_{0}}|\nabla\eta|^{2}e^{-f}d\nu \\
\leq &\frac{a_{5}t_{0}}{R^{2}}\int_{\Omega}\left(h^{\alpha_{0}+t_{0}}\eta^{\frac{2\alpha_{0}+2t_{0}}{\alpha_{0}+t_{0}+1}}(e^{-f})^{\frac{\alpha_{0}
+t_{0}}{\alpha_{0}+t_{0}+1}}\right)\left((e^{-f})^{\frac{1}{\alpha_{0}+t_{0}+1}}\right)d\nu \\
\leq&\frac{a_5t_0}{R^2}\left(\int_\Omega h^{\alpha_0+t_0+1}\eta^2e^{-f}d\nu\right)^{\frac{\alpha_0+t_0}{\alpha+t_0+1}}\left(\int_{\Omega}e^{-f}d\nu\right)^{\frac1{\alpha_0+t_0+1}} \\
=&\frac{a_5t_0}{R^2}\left(\int_\Omega h^{\alpha_0+t_0+1}\eta^2e^{-f}d\nu\right)^{\frac{\alpha_0+t_0}{\alpha+t_0+1}}V_f^{\frac1{\alpha_0+t_0+1}}.
\endaligned
\end{equation}
Then, we use the Young inequality, which is described as follows: for any $\varepsilon>0$,
$$\Theta\Lambda\leq\varepsilon\frac{\Theta^{p}}{p}+\varepsilon^{-\frac{q}{p}}\frac{\Lambda^q}{q},\text{ with } q=\frac{p}{p-1}\text{ relating to }p.$$
We choose $\varepsilon=\frac{\beta_{0}R^{2}}{2a_{5}t_{0}}$, $p=\frac{\alpha_0+t_0+1}{\alpha_0+t_0}$, $q=\alpha_0+t_0+1$, $\Lambda=V_f^{\frac1{\alpha_0+t_0+1}}$ and
$$\Theta=\big(\int_{\Omega}h^{\alpha_{0}+t_{0}+1}\eta^{2}e^{-f}d\nu\big)^{\frac{\alpha_0+t_0}{\alpha+t_0+1}},$$
then we can get
\begin{equation}\label{add3sec3-20}
\aligned
&\frac{a_5t_0}{R^2}\left(\int_\Omega h^{\alpha_0+t_0+1}\eta^2e^{-f}d\nu\right)^{\frac{\alpha_0+t_0}{\alpha+t_0+1}}V_f^{\frac1{\alpha_0+t_0+1}} \\
\leq&\frac{a_{5}t_{0}}{R^{2}}\bigg[\frac{\beta_{0}R^{2}}{2a_{5}t_{0}}\frac{\alpha_0+t_0}{\alpha_0+t_0+1}\bigg(\big(\int_{\Omega}h^{\alpha_{0}+t_{0}
+1}\eta^{2}e^{-f}d\nu\big)^{\frac{\alpha_0+t_0}{\alpha+t_0+1}}\bigg)^{\frac{\alpha+t_0+1}{\alpha+t_0}}\\
&+\left(\frac{\beta_{0}R^{2}}{2a_{5}t_{0}}\right)^{-(\alpha_0+t_0)}\frac{1}{\alpha_0+t_0+1}(V_f^{\frac1{\alpha_0+t_0+1}})^{\alpha_0+t_0+1}\bigg]\\
\leq&\frac{\beta_{0}}{2}\left[\int_{\Omega}h^{\alpha_{0}+t_{0}+1}\eta^{2}e^{-f}d\nu+\left(\frac{2a_{5}t_{0}}{\beta_{0}R^{2}}\right)^{\alpha_{0}+t_{0}+1}V_f\right].
\endaligned
\end{equation}
Hence,  by \eqref{sec3-18} and \eqref{add3sec3-20}, we have
\begin{equation}\label{sec3-21}
\aligned&\left(\int_{\Omega}h^{\frac{m(\alpha_0+t_0)}{m-2}}\eta^{\frac{2m}{m-2}}e^{-f}d\nu\right)^{\frac{m-2}{m}}\\
\leq &\frac{t_0}{a_2}e^{t_0}V_f^{1-\frac2m}R^2\left[\frac{2a_4t_0^2}{R^2}\left(\frac{2a_4t_0^2}{\beta_0R^2}\right)^{\alpha_0+t_0}+\frac{a_5t_0^2}{R^2}\left(\frac{2a_5t_0}{\beta_0R^2}\right)^{\alpha_0+t_0}\right]\\
\leq & a_6^{t_0}V_f^{1-\frac2m}t_0^3\left(\frac{t_0^2}{R^2}\right)^{\alpha_0+t_0},\endaligned
\end{equation}
where the constant $a_6$ depends only on $m$, $\lambda$, $\tau$ and satisfies
$$a_6^{t_0}\geq\frac{2a_4}{a_2}e^{t_0}\left(\dfrac{4a_4}{\beta_0}\right)^{\alpha_0+t_0}+\frac{a_5}{a_2}e^{t_0}\left(\dfrac{4a_5}{\beta_0t_0}\right)^{\alpha_0+t_0}.$$
Here we have used the fact $t_0\geq1.$ Taking $\frac1{\alpha_0+t_0}$ power of the both sides of \eqref{sec3-21} gives
\begin{equation}\label{sec3-22}
\left(\int_{\Omega}(h\eta^{\frac{2}{\alpha_0+t_0}})^{\beta}e^{-f}d\nu\right)^{\frac{1}{\beta}}\leq a_6^{\frac{t_0}{\alpha_0+t_0}}V_f^{^{\frac{1}{\beta}}}t_0^{\frac{3}{\alpha_0+t_0}}\frac{t_0^2}{R^2}\leq a_7V_f^{\frac{1}{\beta}}\frac{t_0^2}{R^2},
\end{equation}
where the constant $a_7$ depends only on $m$, $\lambda$, $\tau$ and satisfies
$$a_7\geq a_6^{\frac{t_0}{\alpha_0+t_0}}t_0^{\frac3{\alpha_0+t_0}}.$$
Since $\eta\equiv1$ in $B_{3R/4}(o)$, we have
$$\left(\int_{B_{3R/4}(o)}h^{\beta}e^{-f}d\nu\right)^{\frac{1}{\beta}}\leq a_7V_f^{\frac{1}{\beta}}\frac{t_0^2}{R^2}.$$
Thus, we obtain the desired estimate.
\end{proof}

Now, we are in the position to give the proof of Theorem \ref{Int-1} by applying the Nash-Moser iteration method.

\begin{proof} Now we let
$$\left\|h^{\frac{\alpha_0+t}{2}}\eta\right\|_{L_f^{\frac{2m}{m-2}}(\Omega)}^2:=\left(\int_{\Omega}|h^{\frac{\alpha_{0}+t}{2}}\eta|^{\frac{2m}{m-2}}
e^{-f}d\nu\right)^{\frac{m-2}{m}}.$$
Suppose that $v$ is a positive solution to the equation \eqref{zjInt1} on a complete $n$-dimensional smooth metric measure space $(M^{n}, g, e^{-f} d\nu)$ with the constants $m>2$, $\lambda$ and $\tau$ which satisfy the condition in Theorem \ref{Int-1} (or Corollary \ref{Int-1-3}), $u= -\ln v$ and $h=|\nabla u|^2$.

Then, we get rid of the first term in \eqref{sec3-14} and thereby obtain
$$\aligned
&\frac{a_2}{t}e^{-t_0}V_f^{\frac{2}{m}}R^{-2}\left(\int_{\Omega}|h^{\frac{\alpha_{0}+t}{2}}\eta|^{\frac{2m}{m-2}}
e^{-f}d\nu\right)^{\frac{m-2}{m}} \\
\leq &\frac{a_4t_0^2}{R^2}\int_\Omega h^{\alpha_0+t}\eta^2e^{-f}d\nu+\frac{a_3}{t}\int_\Omega h^{\alpha_0+t}|\nabla\eta|^2e^{-f}d\nu\\
\leq &\int_{\Omega}\left(\frac{a_4t_0^2}{R^2}\eta^2+\frac{a_3}{t}|\nabla\eta|^2\right)h^{\alpha_0+t}e^{-f}d\nu,
\endaligned$$
which is equivalent to
\begin{equation}\label{sec3-23}\aligned
&\left(\int_{\Omega}|h^{\frac{\alpha_{0}+t}{2}}\eta|^{\frac{2m}{m-2}}e^{-f}d\nu\right)^{\frac{m-2}{m}}\\
\leq&e^{t_0}V_f^{-\frac{2}{m}}\int_{\Omega}\left(\frac{a_4}{a_{2}}t_0^2t\eta^2+\frac{a_3}{a_2}R^{2}|\nabla\eta|^2\right)h^{\alpha_0+
t}e^{-f}d\nu.\endaligned
\end{equation}
Then, we take an increasing sequence $\{\beta_k\}_{k=1}^{\infty}$ such that
$$\beta_1=\beta\quad\text{and}\quad\beta_{k+1}=\frac{m\beta_{k}}{m-2},\quad k= 1, 2, \ldots,$$
and a decreasing sequence $\{r_k\}_{k=1}^{\infty}$ such that
$$r_k=\frac{R}{2}+\frac{R}{4^k},\quad k= 1, 2, \ldots.$$
Let $\Omega_k=B_{r_k}(o)$, then we may choose $\{\eta_k\}_{k=1}^{\infty}\subset C_{0}^{\infty}(B_R(o))$ such that
$$\eta_k\in C_{0}^{\infty}(B_{r_k}(o)),\quad\eta_k=1\text{ in }B_{r_{k+1}}(o)\quad\text{and}\quad|\nabla\eta_k|\leq\frac{C4^k}{R}.$$
By letting $t=t_k$, $\eta=\eta_k$ and $t_k+\alpha_0=\beta_k$ in \eqref{sec3-23}, we derive that
\begin{equation}\label{add3sec3-23}\aligned
&\left(\int_{\Omega_k}|h^{\frac{\alpha_{0}+t_k}{2}}\eta_k|^{\frac{2m}{m-2}}e^{-f}d\nu\right)^{\frac{m-2}{m}}\\
\leq&e^{t_0}V_f^{-\frac{2}{m}}\int_{\Omega_k}\left(\frac{a_4}{a_{2}}t_0^2t_k\eta_k^2+\frac{a_3}{a_2}R^{2}|\nabla\eta_k|^2\right)h^{\alpha_0+
t_k}e^{-f}d\nu\\
\leq&e^{t_0}V_f^{-\frac{2}{m}}\int_{\Omega_k}\left(\frac{a_4}{a_{2}}t_0^2t_k\eta_k^2+\frac{a_3}{a_2}C^{2}16^k\right)h^{\alpha_0+
t_k}e^{-f}d\nu\\
\leq&e^{t_0}V_f^{-\frac{2}{m}}\left(\frac{a_4}{a_{2}}t_0^2t_k+\frac{a_3}{a_2}C^{2}16^k\right)\int_{\Omega_k}h^{\alpha_0+
t_k}e^{-f}d\nu\\
\leq&e^{t_0}V_f^{-\frac{2}{m}}\left(\frac{a_4}{a_{2}}t_0^2\left(t_0+\alpha_0\right)\left(\frac{m}{m-2}\right)^k
+\frac{a_3}{a_2}C^{2}16^k\right)\int_{\Omega_k}h^{\alpha_0+t_k}e^{-f}d\nu\\
\leq&e^{t_0}V_f^{-\frac{2}{m}}\left(\frac{a_4}{a_{2}}t_0^2\left(t_0+\alpha_0\right)16^k
+\frac{a_3}{a_2}C^{2}16^k\right)\int_{\Omega_k}h^{\alpha_0+t_k}e^{-f}d\nu.
\endaligned
\end{equation}
We can find some constant $a_{8}$ such that
\begin{equation}\label{sec3-25}
\left(\int_{\Omega_k}|h^{\frac{\alpha_{0}+t_k}{2}}\eta_k|^{\frac{2m}{m-2}}e^{-f}d\nu\right)^{\frac{m-2}{m}}
\leq\:a_8t_0^316^ke^{t_0}V_f^{-\frac{2}{m}}\int_{\Omega_k}h^{\beta_k}e^{-f}d\nu.
\end{equation}
Then, taking power of $1/\beta_k$ of the both sides of \eqref{sec3-25}, we have
\begin{equation}\label{addsec3-25}
\left(\int_{\Omega_k}|h^{\frac{\alpha_{0}+t_k}{2}}\eta_k|^{\frac{2m}{m-2}}e^{-f}d\nu\right)^{\frac{m-2}{m\beta_{k}}}
\leq\:(a_8t_0^{3}e^{t_0}V_f^{-\frac{2}{m}})^{\frac{1}{\beta_k}}16^{\frac{k}{\beta_k}}\left(\int_{\Omega_k}h^{\beta_k}e^{-f}d\nu\right)^{\frac{1}{\beta_k}}.
\end{equation}
Thus,
$$\left(\int_{\Omega_{k+1}}h^{\beta_{k+1}}e^{-f}d\nu\right)^{\frac{1}{\beta_{k+1}}}
\leq\:(a_8t_0^{3}e^{t_0}V_f^{-\frac{2}{m}})^{\frac{1}{\beta_k}}16^{\frac{k}{\beta_k}}\left(\int_{\Omega_k}h^{\beta_k}e^{-f}d\nu\right)^{\frac{1}{\beta_k}},$$
which implies that
\begin{equation}\label{sec3-26}
\left\|h\right\|_{L_f^{\beta_{k+1}}(\Omega_{k+1})}\leq\left(a_{8}t_{0}^{3}e^{t_{0}}V_{f}^{-\frac{2}{m}}\right)^{\frac{1}{\beta_{k}}}16^{\frac{k}{\beta_{k}}}\|h\|_{L_f^{\beta_{k}}(\Omega_{k})}.
\end{equation}
By iteration we have
\begin{equation}\label{add1sec3-26}
\left\|h\right\|_{L_f^{\beta_{k+1}}(\Omega_{k+1})}\leq\left(a_{8}t_{0}^{3}e^{t_{0}}V_{f}^{-\frac{2}{m}}\right)^{\sum_{k=
1}^\infty\frac{1}{\beta_k}}16^{\sum_{k=1}^\infty\frac{k}{\beta_k}}\|h\|_{L_f^{\beta_{1}}(B_{3R/4}(o))}.
\end{equation}
We note that
$$\sum_{k=1}^\infty\frac{1}{\beta_k}=\frac{m}{2\beta_1}\quad\text{and}\quad\sum_{k=1}^\infty\frac{k}{\beta_k}=\frac{m^2}{4\beta_1},$$
then, by letting $k\rightarrow\infty$ in \eqref{add1sec3-26} we obtain the following inequality:
\begin{equation}\label{sec3-27}
\|h\|_{L_f^\infty(B_{R/2}(o))}\leq\:a_{9}V_{f}^{-\frac{1}{\beta_1}}\|h\|_{L_f^{\beta_1}(B_{3R/4}(o))},
\end{equation}
where
$$a_{9}\geq\left(a_8t_0^3e^{t_0}\right)^{\frac{m}{2\beta_1}}16^{\frac{m^2}{4\beta_1}}.$$
According to \eqref{sec3-15},  we have
\begin{equation}\label{sec3-28}
\|h\|_{L_f^\infty(B_{R/2}(o))}\leq\:a_{10}\frac{(1+\sqrt{K}R)^2}{R^2},
\end{equation}
where $a_{10}=a_{9}a_7c_{m,\lambda ,\tau}^2.$ Recalling $h=|\nabla u|^2$ and $u=-\ln v$, we derive the required estimate. Consequently, we conclude the proof of Theorem \ref{Int-1}.
\end{proof}

\section{Proofs of Corollary \ref{Int-1-4} and Corollary \ref{Int-1-5}}

\textbf{Proof of Corollary \ref{Int-1-4}.}
Under the conditions in Corollary \ref{Int-1-4}, then by Theorem \ref{Int-1} we have that for any $x\in B_{R/ 2}(o) \subset M$,
\begin{equation}\label{sec3-29}
\frac{|\nabla v(x)|}{v(x)}\leq\sup\limits_{B_{R/ 2}(o)}\frac{|\nabla v|}{v}\leq\frac{c(m,\tau, \lambda)}{R}.
\end{equation}
Letting $R\to\infty$ in \eqref{sec3-29}, we get
$$\nabla v(x)=0,\quad\forall x\in M.$$
Therefore, $v$ is a positive constant on $M$.

\textbf{Proof of Corollary \ref{Int-1-5}.}
Under the same conditions in Theorem \ref{Int-1}, let $x$, $y\in B_{R/ 2}(o)$ be any two points with minimal geodesic $\gamma$ connecting them. Then using the gradient estimate \eqref{Int10} and the fact that length$(\gamma)\leq R$, we have
$$\aligned
\ln v(x)-\ln v(y)=&\int_{\gamma}|\nabla\ln v|dt\\
\leq&\int_{\gamma}c(\tau,m,\lambda)\frac{1+\sqrt{K}R}{R}dt\\
\leq&c(\tau,m,\lambda)(1+\sqrt{K}R).
\endaligned$$
Thus, for any $x$, $y\in B_{R/ 2}( o)$ we obtain
$$v(x)\leq e^{c(\tau,m,\lambda)(1+\sqrt{K}R)}v(y).$$

\section{Global gradient estimate}
By the local gradient estimate of positive solution $v$ of the equation \eqref{zjInt1}, we have
$$\sup_{B_{R/2}(o)}\frac{|\nabla v|}{v}\leq c(m,\tau,\lambda)\frac{1+\sqrt{K}R}{R}.$$
If $v$ is a global solution, letting $R\to\infty$ in the above local gradient estimate, we obtain
$$\frac{|\nabla v|}{v}\leq c(m,\tau,\lambda)\sqrt{K},\quad \forall x\in M.$$
In this section, we give a explicit expression of the above $c(m,\tau,\lambda)$.

We have following key lemmas.
\begin{lemma}\label{adsec5-2}
Let $(M^{n}, g, e^{-f} d\nu)$ be a complete $n$-dimensional smooth metric measure space with ${\rm Ric}_{f}^{m}\geq-(m-1)Kg$ where $K$ is a non-negative constant. Let $v$ be a global positive solution of \eqref{zjInt1} in $M$.

(1) If $m$, $\tau$ and $\lambda$ satisfy \eqref{Int5}, we denote
$$y_1=\frac{(m-1)^2}{\tau^2}K,$$
then for any $\delta>0$, $\omega=(h-y_1-\delta)^+$ and $\alpha$ large enough, there are some positive constants $l_1$, $l_2$ depending on $m$, $\tau$, $\lambda$, $K$ and $\delta$ such that
$$\Delta_{f}\omega^{\alpha}\geq2\alpha\omega^{\alpha-1}(l_1\omega-l_2|\nabla\omega|).$$

(2) If $m$, $\tau$ and $\lambda$ satisfy \eqref{Int6}, we denote
$$y_2=\frac{K}{\frac{\tau^2}{2}\left(\frac{\sqrt{2}-m-1}{m-1}+\frac{1}{\tau}\right)\left(\frac{\sqrt{2}+m
+1}{m-1}-\frac{1}{\tau}\right)},$$
then for any $\delta>0$, $\omega=(h-y_2-\delta)^+$ and $\alpha$ large enough, there are some positive constants $l_3$, $l_4$ depending on $m$, $\tau$, $K$ and $\delta$ such that
$$\Delta_{f}\omega^{\alpha}\geq2\alpha\omega^{\alpha-1}(l_3\omega-l_4|\nabla\omega|).$$
\end{lemma}
\begin{proof}
(1) By the Lemma \ref{lem2-2}, if $m$, $\tau$ and $\lambda$ satisfy \eqref{Int5}, there holds
\begin{equation}\label{sec5-3}
\frac{h^{1-\alpha}}{2\alpha}\Delta_fh^{\alpha}\geq\frac{1}{m-1}\tau^2h^2-(m-1)Kh-\frac{a_1}{2}h^{\frac12}|\nabla h|.
\end{equation}
We note that $\nabla w=\nabla h$ in $\{h\geq y_1+\delta\}$, but $\nabla w$ causes a distribution on $\{h=y_1+\delta\}$. Now, if we assume $\alpha>1$, then the distribution caused by $\nabla w$
on $\{h=y_1+\delta\}$ is eliminated by $w$ since $w=0$ on $\{h=y_1+\delta\}$.

Since $h\geq w$, it follows that
\begin{equation}\label{addsec5-3}\aligned
\Delta_fw^{\alpha}=&-\alpha w^{\alpha-1}\langle\nabla w,\nabla f\rangle+\mathrm{div}(\alpha w^{\alpha-1}\nabla w)\\
=&-\alpha w^{\alpha-1}\langle\nabla h,\nabla f\rangle+\mathrm{div}(\alpha w^{\alpha-1}\nabla h)\\
=&\alpha w^{\alpha-1}\Delta_fh+\alpha(\alpha-1)w^{\alpha-2}|\nabla h|^2\\
=&\frac{w^{\alpha-1}}{h^{\alpha-1}}\left(\alpha h^{\alpha-1}\Delta_fh+\alpha(\alpha-1)w^{-1}h^{\alpha-1}|\nabla h|^2\right)\\
\geq&\frac{w^{\alpha-1}}{h^{\alpha-1}}\left(\alpha h^{\alpha-1}\Delta_fh+\alpha(\alpha-1)h^{\alpha-2}|\nabla h|^2\right)\\
=&\frac{w^{\alpha-1}}{h^{\alpha-1}}\Delta_fh^{\alpha}.
\endaligned
\end{equation}
Substituting \eqref{sec5-3} into the \eqref{addsec5-3}, we get
$$\Delta_fw^{\alpha}\geq2\alpha w^{\alpha-1}\left(\frac{1}{m-1}\tau^2h^2-(m-1)Kh-\frac{a_1}{2}h^{\frac12}|\nabla h|\right).$$
Since $h\geq y_1+\delta=\frac{(m-1)^2}{\tau^2}K+\delta$, the above inequality can be rewritten as
$$\aligned
\Delta_fw^{\alpha}\geq&2\alpha w^{\alpha-1}h^{\frac12}\left(h^{\frac12}\left(\frac{1}{m-1}\tau^2h-(m-1)K\right)-\frac{a_1}{2}|\nabla w|\right)\\
\geq&2\alpha w^{\alpha-1}h^{\frac12}\left(h^{\frac12}\frac{\delta\tau^2}{m-1}-\frac{a_1}{2}|\nabla w|\right).
\endaligned$$
There holds
$$y_1+\delta\leq h\leq c(m,\tau,\lambda)\sqrt{K}$$
on $\{h\geq y_1+\delta\}$, thus we have
$$\Delta_fw^{\alpha}\geq2\alpha w^{\alpha-1}(l_1w-l_2|\nabla w|),$$
where $l_1$ and $l_2$ are two positive constants depending on $m$, $K$, $\tau$, $\lambda$ and $\delta$.

(2) If $m$, $\tau$ and $\lambda$ satisfy \eqref{Int6}, it follows from \eqref{add4sec2-40} that
\begin{equation}\label{addsec5-4}
\frac{h^{1-\alpha}}{2\alpha}\Delta_fh^{\alpha}\geq B_{m,\tau}h^2-(m-1)Kh-\frac{a_1}{2}h^{\frac12}|\nabla h|,
\end{equation}
where
$$B_{m,\tau}=\left(\frac{1}{m-1}-\frac{m-1}{2}\left(\frac{m+1}{m-1}-\frac{1}{\tau}\right)^2\right)\tau^2.$$
Then, on $\{h\geq y_2+\delta\}$,
$$h\geq y_2+\delta=\frac{(m-1)K}{B_{m,\tau}}+\delta.$$
So the \eqref{addsec5-4} can be rewritten as
$$\aligned
\Delta_fw^{\alpha}\geq&2\alpha w^{\alpha-1}h^{\frac12}\left(h^{\frac12}\left(B_{m,\tau}h-(m-1)K\right)-\frac{a_1}{2}|\nabla w|\right)\\
\geq&2\alpha w^{\alpha-1}h^{\frac12}\left(h^{\frac12}\delta B_{m,\tau}-\frac{a_1}{2}|\nabla w|\right).
\endaligned$$
There holds
$$y_2+\delta\leq h\leq c(m,\tau,\lambda)\sqrt{K}$$
on $\{h\geq y_2+\delta\}$, so we obtain
$$\Delta_fw^{\alpha}\geq2\alpha w^{\alpha-1}(l_3w-l_4|\nabla w|),$$
where $l_3$ and $l_4$ are two positive constants depending on $m$, $K$, $\tau$ and $\delta$.
\end{proof}
\begin{lemma}\label{adsec5-3}
Let $(M^{n}, g, e^{-f} d\nu)$ be a complete $n$-dimensional smooth metric measure space with ${\rm Ric}_{f}^{m}\geq-(m-1)Kg$ where $K$ is a non-negative constant. Let $v$ be a global positive solution of \eqref{zjInt1} in $M$.
For some $y>0$, we let $w=(h-y)^+$. If $w$ satisfy the following inequality
\begin{equation}\label{addsec5-5}
\Delta_fw^{\alpha}\geq2\alpha w^{\alpha-1}(l_3w-l_4|\nabla w|),
\end{equation}
where $l_3$ and $l_4$ and $\alpha$ are some positive constants, then $w=0$, i.e., $h\leq y$.
\end{lemma}
\begin{proof}
Let $\eta\in C_0^\infty (M, \mathbb{R})$ be a cut-off function to be determine later. For constant $\gamma>1$, we choose $w^\gamma\eta^2$ as test function. Now, we multiply the both sides of \eqref{addsec5-5} by $w^\gamma\eta^2$, integrate then the obtained inequality over $M$ to get
$$\int_{M}(w^\gamma\eta^2)\Delta_fw^{\alpha}e^{-f}d\nu\geq\int_{M}2\alpha w^{\alpha+\gamma-1}(l_3w-l_4|\nabla w|)\eta^2e^{-f}d\nu.$$
Using the integration by parts, we have
$$\aligned
&-\int_{M}\alpha\gamma w^{\alpha+\gamma-2}|\nabla w|^2\eta^2e^{-f}d\nu-\int_{M}2\alpha w^{\alpha+\gamma-1}\langle\nabla w,\nabla\eta\rangle\eta e^{-f}d\nu\\
\geq&\int_{M}2\alpha w^{\alpha+\gamma-1}(l_3w-l_4|\nabla w|)\eta^2e^{-f}d\nu.
\endaligned$$
Since
$$\langle\nabla w,\nabla\eta\rangle\geq-|\nabla w||\nabla\eta|,$$
we conclude that
$$\aligned
&\int_{M}2l_3w^{\alpha+\gamma}\eta^2e^{-f}d\nu+\int_{M}\gamma w^{\alpha+\gamma-2}|\nabla w|^2\eta^2e^{-f}d\nu\\
\leq&\int_{M}2l_4w^{\alpha+\gamma-1}|\nabla w|\eta^2e^{-f}d\nu+\int_{M}2w^{\alpha+\gamma-1}|\nabla w||\nabla\eta|\eta e^{-f}d\nu.
\endaligned$$
Now applying Young's inequality, we get
$$2l_4w^{\alpha+\gamma-1}|\nabla w|\eta^2\leq l_4w^{\alpha+\gamma-2}\frac{|\nabla w|^2}{\epsilon}\eta^2+\epsilon l_4 w^{\alpha+\gamma}\eta^2;$$
$$2w^{\alpha+\gamma-1}|\nabla w||\nabla\eta|\eta\leq w^{\alpha+\gamma-2}\frac{|\nabla w|^2}{\epsilon}\eta^2+\epsilon w^{\alpha+\gamma}|\nabla\eta|^2.$$
Now we choose $\epsilon$ such that
$$\frac{l_4+1}{\epsilon}=\gamma,$$
then we obtain
$$\int_{M}2l_3w^{\alpha+\gamma}\eta^2e^{-f}d\nu\leq\int_{M}\epsilon w^{\alpha+\gamma}|\nabla\eta|^2e^{-f}d\nu+\int_{M}\epsilon l_4 w^{\alpha+\gamma}\eta^2e^{-f}d\nu.$$
We choose $\gamma$ large enough such that $\epsilon l_4<l_3$, we have
$$\int_{M}l_3w^{\alpha+\gamma}\eta^2e^{-f}d\nu\leq\int_{M}\epsilon w^{\alpha+\gamma}|\nabla\eta|^2e^{-f}d\nu.$$
Now, if $w\neq0$, without loss of generality we may suppose that $w\mid_{B_1}\neq0$ for some geodesic
ball $B_1$ with radius 1, we observe that there always holds true
$$\int_{B_1}e^{-f}d\nu>0.$$
We use $B_k$ to represent the geodesic ball with a radius of $k$. Then we may choose $\eta\in C_{0}^{\infty}(B_{k+1})$ such that
$$\eta\equiv1\text{ in }B_{k}\quad\text{and}\quad|\nabla\eta|\leq4\text{ in }B_{k+1}.$$
Then, we derive that
$$\aligned
\int_{B_{k+1}}16\epsilon w^{\alpha+\gamma}e^{-f}d\nu\geq&\int_{B_{k+1}}l_3w^{\alpha+\gamma}e^{-f}d\nu\\
\geq&\int_{B_k}l_3w^{\alpha+\gamma}e^{-f}d\nu.
\endaligned$$
By iteration on $k$, we get
\begin{equation}\label{sec5-5}
\int_{B_k}w^{\alpha+\gamma}e^{-f}d\nu\geq\left(\frac{\hat{C}_1}{\epsilon}\right)^k\int_{B_{1}}w^{\alpha+\gamma}e^{-f}d\nu.
\end{equation}
where $\hat{C}_1$ is positive constant.
Since $w$ is uniform bounded, we have
\begin{equation}\label{sec5-6}
\hat{C}_2^{\alpha+\gamma}V_f(B_k)\geq\int_{B_k}w^{\alpha+\gamma}e^{-f}d\nu
\end{equation}
and
\begin{equation}\label{sec5-7}
\int_{B_{1}}w^{\alpha+\gamma}e^{-f}d\nu\geq \hat{C}_3^{\alpha+\gamma}V_f(B_{1}),
\end{equation}
where $\hat{C}_2$ and $\hat{C}_3$ are positive constant, and $V_{f}(B_{k})=\int_{B_{k}}e^{-f}d\nu$.
By volume comparison theorem, we have
\begin{equation}\label{sec5-8}
\frac{V_f(B_k)}{V_f(B_1)}\leq e^{(m-1)k\sqrt K}.
\end{equation}
Let $\hat{C}_{4}=\frac{\hat{C}_2}{\hat{C}_3}$, then plugging \eqref{sec5-6}, \eqref{sec5-7}, and \eqref{sec5-8} into \eqref{sec5-5}, we infers
\begin{equation}\label{sec5-9}
\hat{C}_{4}^{\alpha+\gamma}e^{(m-1)k\sqrt K}\geq e^{k\ln\frac{\hat{C}_1}{\epsilon}}.
\end{equation}
We choose $\gamma$ such that
$$\ln\frac{\hat{C}_1}{\epsilon}>2(m-1)\sqrt{K}+2,$$
if $k$ is chosen sufficiently large, \eqref{sec5-9} can not hold. Hence, $w\equiv0$. So we finish the proof of this lemma.
\end{proof}
Now, we are in the position to give the proof of Theorem \ref{adsec5-1}.
\begin{proof}
By Lemma \ref{adsec5-2} and Lemma \ref{adsec5-3}, we can infer the Theorem \ref{adsec5-1}. If $m$, $\lambda$, $\tau$ satisfy \eqref{Int5} or \eqref{Int6}, then by Lemma \ref{adsec5-2}, for any $\delta>0$ and $w=(h-y_i-\delta)^+$, $i=1,2$, we obtain
$$\Delta_fw^{\alpha}\geq2\alpha w^{\alpha-1}(l_3w-l_4|\nabla w|).$$
According to Lemma \ref{adsec5-3}, we obtain $w\equiv0$, i.e., $h\leq y_i+\delta$. Since $\delta$ can be arbitrary small, it is clear that $h\leq y_i$. Since
$$h=|\nabla u|^2=|(\nabla v)/v|^2,$$
thus, we get
$$\frac{|\nabla v|}{v}\leq\sqrt{y_i}.$$
Hence, we finish the proof of Theorem \ref{adsec5-1}.
\end{proof}

\section{Declarations}
\textbf{Data availibility}: Data sharing is not applicable to this article as no new data were created or analyzed in this study.

\textbf{Conflict of interest}: The authors declare that they have no conflict of interest.

\textbf{Author Contributions}: All authors have contributed equally to this work. All authors reviewed the manuscript.

\vskip 6mm
\noindent{\bf Acknowledgements}

\noindent
The author Y. Wang
is supported partially by National Key Research and Development projects of China (Grant
No. 2020YFA0712500).

\end{document}